\documentclass{amsproc}

\usepackage{amsmath,amscd,amssymb}
\usepackage[mathscr]{eucal}
\usepackage[all]{xy}

\newtheorem{theorem}{Theorem}[section]
\newtheorem{lemma}[theorem]{Lemma}
\newtheorem{proposition}[theorem]{Proposition}
\newtheorem{corollary}[theorem]{Corollary}

\theoremstyle{definition}

\newtheorem{example}[theorem]{Example}

\theoremstyle{remark}
\newtheorem{remark}[theorem]{Remark}

\numberwithin{equation}{section}

\begin{document}

\title[On rates in Euler's formula for $C_0$-semigroups]
{On rates in Euler's formula for $C_0$-semigroups}

\author{Alexander Gomilko}
\address{Faculty of Mathematics and Computer Science\\
Nicolas Copernicus University\\
ul. Chopina 12/18\\
87-100 Toru\'n, Poland
}

\email{gomilko@mat.umk.pl}

\author{Yuri Tomilov}
\address{Faculty of Mathematics and Computer Science\\
Nicolas Copernicus University\\
ul. Chopina 12/18\\
87-100 Toru\'n, Poland\\
and
Institute of Mathematics\\
Polish Academy of Sciences\\
\'Sniadeckich 8\\
00-956 Warszawa, Poland
}

\email{tomilov@mat.umk.pl}
\dedicatory{To Fritz Gesztesy on the occasion of his sixtieth anniversary, with admiration}

\thanks{This work was completed with the support of the NCN grant
 DEC-2011/03/B/ST1/00407.}



\date{\today}

\begin{abstract}
By functional calculus methods,
we obtain optimal convergence rates
in Euler's approximation formula for $C_0$-semigroups
restricted to ranges of
generalized Stieltjes functions.
Our results include a number of partial cases studied in the literature
 and cannot essentially be improved.
\end{abstract}

\maketitle

\section{Introduction}
Let $-A$ be the generator of a bounded $C_0$-semigroup  $(e^{-tA})_{t\ge 0}$ on a Banach space $X$.
Then the abstract Cauchy problem
\begin{equation} \label{cauchy}
\left\{ \begin{array}{ll}
x'(t) = - A x(t) , & \qquad t \ge 0 ,\\[2mm]
x(0) = x_0 , & \qquad x_0 \in X,
\end{array} \right.
\end{equation}
is well-posed and all its mild solutions are given by the formula
$$
x(t)=e^{-tA}x_0, \qquad t\ge 0, \quad x_0 \in X.
$$
However, even if $A$ is bounded, the exponential function  $e^{-tA}$ can hardly be given in an explicit form. Thus it is of importance
for applications to find approximation formulas for  $e^{-tA}$ amenable for the purposes of numerical analysis, e.g. formulas involving rational functions of $A.$
Starting from the pioneering works of Hersh and Kato \cite{HK} and
P. Brenner and V. Thom\'ee \cite{BT79}, the methods of
Hille-Phillips
functional calculus have played an important role in the theory of rational approximations of $C_0$-semigroups,
see e.g. \cite[Introduction and Chapter 1]{KovDis} for a survey.

In this paper, we extend further the functional calculus approach by replacing the ``conventional'' Hille-Phillips functional calculus by the extended
Hille-Phillips functional calculus and then restricting ourselves to the important part of the extended Hille-Phillips calculus given by generalized Stieltjes functions. This approach proved to be quite successful in dealing with rates in mean ergodic theorems for continuous and discrete operator semigroups, see \cite{GHT1}, \cite{GHT} and \cite{GT}.
To demonstrate the power of our approach we consider the simplest semigroup approximation known as Euler's exponential formula
or the Post-Widder inversion formula. The approximation arises when the abstract Cauchy problem \eqref{cauchy}
is time-discretized by the so-called  Euler backward method and it can be defined as
\[
E_{n,t}(A)x:=\left(1+\frac{t}{n}A\right)^{-n}x, \qquad x \in X, \quad n \in \mathbb N, \quad t>0.
\]
It is well known that for every $x \in  X,$
\begin{equation}\label{euler}
e^{-tA}x=\lim_{n \to \infty} E_{n,t}(A)x
\end{equation}
uniformly for $t$ in compact intervals of positive semi-axis.
Thus a natural question is whether it is possible to quantify the convergence in \eqref{euler}.
It is easy show that, in general, there is no rate of convergence in \eqref{euler}
uniform with respect to all elements $x \in X.$
However such a rate does exist when the elements are taken from the domain of an appropriate function of $A$, e.g. a power function.
The next theorem surveys  known results on the rates of convergence of Euler's formula in this case.
Denote \[
\Delta_{n,t}(A):=E_{n,t}(A)-e^{-tA}, \qquad n \in \mathbb N, \quad t>0.
\]
\begin{theorem}\label{survey}
Let $-A$ be the generator of a bounded $C_0$-semigroup  $(e^{-tA})_{t\ge 0}$ on a Banach space $X.$
\begin{itemize}
\item [(i)] \cite[Theorem 4]{BT79} There exists $c >0$ such that for any $n \in \mathbb N$ and $t>0,$
\begin{equation*}
\|\Delta_{n,t}(A)x\|\le c \left(\frac{t}{\sqrt{n}}\right)^2 \, \|A^2x\|,\qquad x\in {\rm dom} \,(A^2);
\end{equation*}
\item [(ii)] \cite[Theorem 1.7]{Fiory} There exists $c> 0$ such that for any $n \in \mathbb N$ and $t>0,$
\begin{equation*}
\|\Delta_{n,t}(A)x\|\le
c\frac{t}{\sqrt{n}}\, \|Ax\|,\qquad x\in {\rm dom}\,(A);
\end{equation*}
\item [(iii)]\cite[Corollary 4.4]{Kovacs07} There exists $c>0$ such that for any $n \in \mathbb N,$  $t>0$ and
$0< \alpha\le 2,$
\begin{equation}\label{inter}
\|\Delta_{n,t}(A)x\|\le c \left(\frac{t}{\sqrt{n}}\right)^{\alpha}
\|x\|_{\alpha,2,\infty},\;\;t\ge 0,\;\;n\in \mathbb N,
\;\;
x\in X_{\alpha,2,\infty},
\end{equation}
where the Banach space $X_{\alpha,2,\infty}$ (called a Favard space) is defined as
\[
X_{\alpha,2,\infty}:=\left\{x\in X:\, \|x\|_{\alpha,2,\infty}:=
\|x\|+\sup_{t>0}\,\frac{\|(e^{-tA}-I)^2x\|}{t^\alpha}<\infty\right\}.
\]
\end{itemize}
\end{theorem}
Some comments concerning the last result are in order.
Note that by \cite[Theorem 4.3]{Komatsu} (see also \cite[Theorem 11.3.5]{Mart1}) if $\alpha\in (0,2)$
then $X_{\alpha,2,\infty}$ coincides
with  Komatsu's (Banach) space
 $D_\infty^\alpha=D_{\infty,2}^\alpha, $
\begin{equation}\label{norm}
D_\infty^\alpha:= \left\{x\in X:\,\|x\|_{D_\infty^\alpha}:=\|x\|+\sup_{\lambda>0}\,\lambda^\alpha
\|[A(A+\lambda)^{-1}]^2 x\|<\infty \right\},
\end{equation}
in the sense that $X_{\alpha,2,\infty}=D_\infty^\alpha$ as sets and the norms are equivalent.
On the other hand,
${\rm dom}\,(A^2)\subset X_{2,2,\infty}$
and, by
\cite[Proposition 2.8]{Komatsu},
$
{\rm dom}\,(A^\alpha)$ is embedded continuously in $D_\infty^\alpha,$ $\alpha\in (0,2).$ However, there are examples (see e.g.
\cite[p. 340]{Komatsu1966})
showing the the inclusion ${\rm dom}\,(A^\alpha) \subset D_\infty^\alpha$ is in general strict.

Thus, \eqref{inter} implies that  there exists $c>0$ such that for any $\alpha\in (0,2],$  $n\in \mathbb N$, and $t>0,$
\begin{equation}\label{kk2}
\|\Delta_{n,t}(A)x\|\le c\left(\frac{t}{\sqrt{n}}\right)^\alpha \|x\|_{{\rm dom}\,(A^\alpha)},\quad x\in {\rm dom}\,(A^\alpha),
\end{equation}
 where
 \[
\|x\|_{{\rm dom}\,(A^\alpha)}:=\|x\|+\|A^\alpha x\|, \qquad x \in {\rm dom}\, (A^\alpha).
\]
(Remark that \cite[Corollary 4.4]{Kovacs07}  only states  that
\eqref{inter} implies  \eqref{kk2} for $\alpha=1,2$.)

We should also emphasize that the results mentioned in Theorem \ref{survey} are in fact partial cases of more general statements on convergence rates for $A$-stable and stable rational approximations of $e^{-At}$ obtained in \cite{BT79}, \cite{Fiory} and \cite{Kovacs07}.
(For similar results see also \cite{Hassan}.) In this paper we consider a very particular case of Euler's approximation  but our results are more general and complete (see also a remark at the end of this section).
  The main problem addressed in this paper is the characterization of  decay rates for
$
\Delta_{n,t}(A)x, x\in {\rm ran}\,(f(A)),
$
where $f$ is a generalized Stieltjes function
of the class $\mathcal{S}_2$.

In particular, we extend Theorem \ref{survey} substantially by replacing power functions  with reciprocals of generalized Stieltjes functions. As a corollary, we are also able to improve Theorem \ref{survey} by showing that
there exists $c>0$ such that for any $n \in \mathbb N$, $t>0,$ and $\alpha\in (0,2],$
\[
\|\Delta_{n,t}(A)x\|\le c \left(\frac{t}{\sqrt{n}}\right)^\alpha \|A^\alpha x\|,\qquad x\in {\rm dom}\,(A^\alpha).
\]
 This result does not hold for $\alpha>2$  as it is explained in  Remark \ref{remarkonalpha}.
We also show that  \eqref{inter} is an easy
consequence of our main result and so it is possible to avoid  interpolation theory used in \cite{Kovacs07}.
Moreover, we prove that our estimates of convergence rates are optimal.

We believe that our method will be fruitful for more general rational approximations as well. However, being confined by space limits, we present only its sample which nevertheless is significant enough to be of value as for semigroup theory so for numerical analysis.

\section{Preliminaries and notations }
The following elementary integrals will be used frequently throughout the paper:

\begin{align}
\frac{1}{(n-1)!}\int_0^\infty s^{n-1} e^{-s}\,ds=&1,
\label{int1}
\\
\int_0^\infty s^{n-1} e^{-s} (1-s/n)\,ds=&0,
\label{zero}
\\
\frac{1}{(n-1)!}\int_0^\infty s^{n-1} e^{-s}(1-s/n)^2\,ds
=&\frac{1}{n},
\label{int20}
\end{align}
where $n\in \mathbb N$.

To simplify our presentation we introduce the next notation:
\begin{align*}
e_t(z):=&e^{-tz},\;\;
r(z):=\frac{1}{1+z},\;\;
r_{n,t}(z):=r^n(zt/n),\\
\Delta_{n,t}(z):=&r_{n,t}(z)-e_t(z),\;\;n\in\mathbb N,\;\;t>0,\;\; h:=\frac{t}{n}.
\end{align*}
Thus, in particular, by (\ref{int1}) we have for $z \in \mathbb C_+:$
\begin{equation}\label{rn}
(n-1)! \, r_{n,t}(z)=\frac{1}{h^n}\int_0^\infty s^{n-1}e^{-h^{-1}s} e^{-zs}\,ds
=\int_0^\infty s^{n-1}e^{-s} e^{-zst/n}\,ds.
\end{equation}

For a closed linear operator $A$ on a complex Banach space $X$ we
denote by ${\rm dom}\,(A),$ ${\rm ran}\,(A)$ and $\sigma(A)$ the
{\em domain}, the {\em range},  and the {\em
spectrum} of $A$, respectively, and let  $\overline{{\rm ran}}\,(A)$ stand for the norm-closure of the range.
 The space of bounded linear operators
on $X$ is denoted by ${\rm L}(X)$.
Finally, we let
\[
\mathbb C_{+}=\{z\in \mathbb C:\,{\rm Re}\,z>0\},\quad \mathbb R_+=[0,\infty).
\]

\section{Functional calculus}

In this subsection we recall definition and basic properties of
functional calculus useful for the sequel.

Let ${\rm M}(\mathbb R_+)$ be a Banach algebra of bounded Radon measures on $\mathbb R_+.$
Define the {\em Laplace transform} of  $\mu \in {\rm M}(\mathbb R_+)$  as
\[ (\mathcal L\mu)(z) := \int_0^{\infty}e^{-sz} \, \mu(d{s}),
\qquad  z \in \mathbb C_+.
\]
Note that the space
\[
{\rm A}^1_+(\mathbb C_+) := \{ \mathcal L\mu : \mu \in {\rm M}(\mathbb R_+)\}
\]
is a commutative Banach algebra with pointwise multiplication and with the
norm
\begin{equation}\label{mmm}
\|\mathcal L \mu\|_{{\rm A}^1_+(\mathbb C_+)} := \|\mu\|_{{\rm M}(\mathbb R_+)} = |{\mu}|(\mathbb R_+),
\end{equation}
where $|\mu|(\mathbb R_+)$ stands for the total variation of $\mu$ on $\mathbb R_+.$
Moreover, the Laplace transform
\[ \mathcal L : {\rm M}(\mathbb R_+) \mapsto {\rm A}^1_+(\mathbb C_+)
\]
is an isometric isomorphism.

Let $-A$ be the generator of a bounded $C_0$-semigroup
$(e^{-tA})_{t\ge 0}$ on a Banach space $X$. Then the mapping
\begin{eqnarray*}
{\rm A}^1_+(\mathbb C_+) &\mapsto& {\rm L}(X),\\
 H(\mathcal L {\mu})x &:= & \int_0^{\infty} e^{-sA}x\, \mu(d{s}), \quad x \in X,
\end{eqnarray*}
defines a continuous algebra homomorphism such that
\begin{equation}\label{hillestimate}
 \| H(\mathcal L {\mu})\| \le \sup_{t \ge 0} \|e^{-tA}\| |\mu|(\mathbb R_+).
\end{equation}
The homomorphism is called the {\em Hille-Phillips} (HP-) functional calculus
for $A$, and we set
\begin{equation*}
g(A)= H(\mathcal L {\mu}) \qquad \text{if} \quad g=\mathcal L \mu.
\end{equation*}
Basic properties of the Hille-Phillips functional calculus can be found in  \cite[Chapter XV]{HilPhi}.

The HP-calculus has an automatic extension to a  function
class much larger then ${\rm A}^1_+(\mathbb C_+)$. Let us recall how this extension is
constructed: if $f: \mathbb C_+ \to \mathbb C$ is holomorphic such that there exists
 $e\in {\rm A}^1_+(\mathbb C_+)$ with $ef \in {\rm A}^1_+(\mathbb C_+)$ and the
operator $e(A)$ is injective, then one defines
\begin{align*}\label{HP}
{\rm dom}\, (f(A)):=&\{x \in X :
(ef)(A)x \in {\rm ran}\,(e(A)) \}\\
f(A) :=& e(A)^{-1} \, (ef)(A).
\end{align*}
Such $f$ is called
{\em regularizable}, and $e$ is called a {\em regularizer} for $f$. This
 definition of $f(A)$ does not depend on the
choice of $e$ and $f(A)$ is a closed
 operator on $X$. The set of all
regularizable functions $f$ constitute an algebra ${\mathcal A}$ depending on $A$
(see e.g. \cite[p. 4-5]{Ha06} and \cite[p. 246-249]{deLau95}).
We call the mapping
\[
\mathcal A \ni f \mapsto  f(A)
\]
 the {\em extended Hille--Phillips calculus} for $A$.
The  next {\em
product rule} of the extended Hille-Phillips  calculus (see e.g. \cite[Chapter 1]{Ha06}) will be crucial for the sequel: {\em if $f$ is regularizable and $g\in {\rm A}^1_+(\mathbb C_+)$,
then}
\begin{equation}\label{hpfc.e.prod}
g(A)f(A)\subset f(A) g(A) = (fg)(A),
\end{equation}
where products of operators have their natural domains.
From \eqref{hpfc.e.prod} it
follows that
if $f$ is regularizable and $e$ is a regularizer, then
\begin{equation}\label{domian}
{\rm ran}\,(e(A))\subset {\rm dom}\,(f(A)).
\end{equation}

\section{Generalized Stieltjes functions}

Our considerations will rely on the notion of generalized Stieltjes function.
We say that a function
$f:(0,\infty)\mapsto [0,\infty)$ is {\it generalized Stieltjes} of order $\alpha>0$ if it can be written as
\[
f(z)=a+\int_0^\infty \frac{\mu(d\tau)}{(z+\tau)^\alpha},\qquad z>0,
\]
where $a\ge 0$ and $\mu$ is a positive Radon measure on $[0,\infty)$
satisfying
\[
\int_0^\infty\frac{\mu(d\tau)}{(1+\tau)^\alpha}<\infty.
\]
Observe that if $f$ is generalized Stieltjes (of any positive order), then $f$ admits an (unique) analytic
extension to $\mathbb C \setminus (-\infty,0]$
which will be identified with $f$ and denoted by the same symbol.
The class of generalized Stieltjes functions of order $\alpha$ will be denoted by $\mathcal{S}_\alpha.$
In this terminology, Stieltjes functions constitute precisely  the class $\mathcal{S}_1$ of generalized
Stieltjes functions of order $1,$
and we will write $\mathcal{S}$ in place of $\mathcal{S}_1$ to denote the class of Stieltjes functions
thus using an established notation.
Note that  $\mathcal{S}\subset \mathcal{S}_2,$ and moreover $\mathcal{S}\cdot \mathcal{S}\subset \mathcal{S}_2$.
Since for every $\alpha \in (0,2]$ one has $z^{-\alpha} \in \mathcal{S}_{\alpha}$ and $\mathcal{S}_\alpha\subset \mathcal{S}_2,$
it clearly follows that $z^{-\alpha} \in \mathcal{S}_2$ for every $\alpha \in (0,2].$
For these as well as many other properties of generalized Stieltjes functions see  \cite{Karp}. A very informative
discussion of Stieltjes functions is contained in \cite[Chapter 2]{SchilSonVon2010}.

We will also need a subclass  $\tilde{\mathcal{S}}_2$ of $\mathcal{S}_2$
consisting of products of Stieltjes functions:
\[
\tilde{\mathcal{S}}_2:=\{f=f_1\cdot f_2:\,\, f_1, f_2\in \mathcal{S}\}.
\]
Observe that the implication
$f_1,f_2\in \mathcal{S}$ $\Rightarrow$ $f_1^{1/2}\cdot f_2^{1/2} \in \mathcal{S}$ (see
\cite[Proposition 7.10]{SchilSonVon2010}) yields
\[
\tilde{\mathcal{S}}_2=\{f=f_0^2:\;\;f_0\in \mathcal{S}\}.
\]

We can {\it define} the class of complete Bernstein functions $\mathcal{CBF}$ as $\mathcal{CBF}:=\{zf : f \in \mathcal{S}\}.$
An important link between the classes of Stieltjes and complete Bernstein functions is provided by the fact that $f \in  \mathcal{CBF}, f \neq 0,$
if and only if $1/f \in \mathcal S,$ see e.g. \cite[Theorem 7.3]{SchilSonVon2010}.

Let now $-A$ be the generator of a bounded $C_0$-semigroup on a Banach space $X.$
By \cite[Lemma 2.5]{GHT}
any complete Bernstein function  is regularizable by $1/(1+z)\in  {\rm A}^1_+(\mathbb C_+).$
Thus if $A$ is injective  then
every $f \in \mathcal{S}$ is regularizable by
$z/(1+z)\in  {\rm A}^1_+(\mathbb C_+).$
The next proposition shows, in particular, that functions from $\mathcal{S}_2$ (and then from $\tilde{\mathcal{S}}_2$) are regularizable as well
and identifies cores of the corresponding operators. To deal with densely defined operators we assume below that the range of $A$ dense.
Note that under this condition, for every $x \in X,$
\begin{equation*}
\lambda (\lambda +A)^{-1}x \to 0, \qquad \lambda \to 0+,
\end{equation*}
and therefore $A$ is also injective (see e.g. \cite[p. 261]{ABHN}).
\begin{proposition}\label{regS}
Let  $-A$ be the generator of a bounded $C_0$-semigroup on a Banach space $X,$
and $\overline{{\rm ran}}(A)=X.$
\begin{itemize}
\item [(i)]
If $f\in \mathcal{S}_2,$ then $f$ is regularizable by
$e(z)=z^2/(1+z)^2\in  {\rm A}^1_+(\mathbb C_+),$ and thus belongs to the extended Hille-Phillips calculus.
Moreover,
${\rm ran}\,(A^2)$
is a core for $f(A).$
\item [(ii)] If
 $f\in \tilde{\mathcal{S}}_2, f \neq 0,$ then $1/f$  is regularizable by
$e(z)=1/(1+z)^2\in  {\rm A}^1_+(\mathbb C_+).$  Hence $1/f$ belongs to the extended Hille-Phillips calculus
and, moreover, ${\rm dom}\,(A^2)$ is a core for $(1/f)(A)$.
\end{itemize}
\end{proposition}
\begin{proof}
To prove (i) note that, since $A^2(I+A)^{-2}$ is injective, a
holomorphic function $f:\mathbb C_+ \to \mathbb C$ is
regularizable by $e$ if and only if  $e f \in  {\rm A}^1_+(\mathbb
C_+).$ Without loss of generality, we can assume that $f\in
\mathcal{S}_2$ is of the form
\begin{equation}\label{S1}
f(z)=\int_0^\infty \frac{\mu(d\tau)}{(z+\tau)^2},\qquad z \in
\mathbb C_+.
\end{equation}
Then
\begin{equation}\label{SS}
f(z)= \int_0^\infty e^{-zs} s\int_0^\infty
e^{-s\tau}\,\mu(d\tau)\,ds,\qquad z \in \mathbb C_+.
\end{equation}
 Since
\begin{align*}
f(z)=&\int_0^\infty e^{-zs} s \int_0^\infty e^{-st}\mu(dt)\,ds\\
=&\int_0^1 e^{-zs} s \int_0^\infty e^{-st}\mu(dt)\,ds
+\int_1^\infty e^{-zs} s \int_0^\infty e^{-st}\mu(dt)\,ds
\end{align*}
it is enough to prove that the second term above is regularizable by $e$.

To this aim note that
\begin{align*}
&\int_1^\infty e^{-zs} s \int_0^\infty e^{-st}\mu(dt)\,ds
\\
=&\frac{e^{-z}}{z} \int_0^\infty e^{-t}\mu(dt)
+\frac{1}{z}\int_1^\infty e^{-sz} \int_0^\infty (1-st) e^{-st}\mu(dt)\,ds
\\
=&\frac{e^{-z}}{z} \int_0^\infty e^{-t}\mu(dt)
+\frac{e^{-z}}{z^2} \int_0^\infty (1-t) e^{-t}\mu(dt)\\
+&\frac{1}{z^2}\int_1^\infty  e^{-zs}\int_0^\infty t(st-2) e^{-st}\mu(dt)\,ds.
\end{align*}
The first two terms in the last display are clearly regularizable by $e$. Let us show that the third term
is regularizable by $e$ too.
Since
\begin{equation*}
\int_1^\infty \left|\int_0^\infty t(ts-2) e^{-st}\,\mu(dt)\right|\,ds\le
\int_0^\infty \int_1^\infty t(ts+2) e^{-st}\,ds\, \mu(dt),
\end{equation*}
and
\begin{equation*}
\int_1^\infty t(ts+2) e^{-st}\,ds=(t+3)e^{-t},
\end{equation*}
it follows that there exists $c>0$ such that
\begin{align*}
\int_1^\infty \left|\int_0^\infty t(ts-2) e^{-st}\,\mu(dt)\right|\,ds\le&
\int_0^\infty (t+3)e^{-t}\,\mu(dt)
\\
\le& c\int_0^\infty \frac{\mu(dt)}{(1+t)^2}\,<\infty.
\end{align*}
Thus, the function
\begin{equation*}
s \mapsto \int_0^\infty t(st-2) e^{-st}\mu(dt)
\end{equation*}
is integrable on $[1,\infty).$ This shows that
\begin{equation*}
z \mapsto \int_1^\infty e^{-zs} s \int_0^\infty e^{-st}\mu(dt)\,ds
\end{equation*}
is regularizable by $e,$ and yields $e f \in {\rm A}^1_+(\mathbb C_+)$.

Moreover, by
\eqref{domian} we have
\[
{\rm ran}\,(A^2)={\rm ran}\, (A^2(I+A)^{-2})\subset{\rm dom}\,(f(A)).
\]
To prove that ${\rm ran}\, (A^2)$ is a a core for $f(A)$ note that if
$e_\epsilon(A)=A^2(\epsilon+ A)^{-2},$ $\epsilon>0,$ then $e_\epsilon(A) x \to x$ for every
$x \in X$ as  $\epsilon \to 0.$ Since $e_\epsilon(z)\in  {\rm A}^1_+(\mathbb C_+)$  for each $\epsilon>0$,  the product rule \eqref{hpfc.e.prod}
implies that if $x \in {\rm dom}\,
(f(A))$ and $f(A)x = y$ then $f(A)e_\epsilon(A)x = e_\epsilon(A)y.$
As ${\rm ran}\,(e_\epsilon(A))= {\rm ran}\,(A^2),$ $\epsilon>0$, the statement follows.

Let us prove (ii) now. Set $g(z)=1/f(z),$ $z >0$. As the reciprocal of a nonzero complete Bernstein function is a Stieltjes function,  $g$ is a product of two complete Bernstein functions.
Then $(1+z)^{-2} g \in  {\rm A}^1_+(\mathbb C_+)$ and, since $(1+A)^2$ is injective,
the function  $g$ is regularizable by $1/(1+z)^2.$
Hence \eqref{domian} yields
\[
{\rm dom}\, (A^2)={\rm ran}\,((1+A)^{-2})\subset {\rm dom}\,(g(A)).
\]
Furthermore, if
$e_\epsilon(A)=(1+\epsilon A)^{-2},$ $\epsilon>0,$ then $e_\epsilon(A) x \to x$ for every
$x \in X$ as  $\epsilon \to 0.$
Arguing as in the proof of (i), we infer that  ${\rm dom}\,(A^2)$ is a core for $g(A)$.
\end{proof}

\begin{remark}\label{zzz}
Let  $f\in \mathcal{S}_2$ be of the form \eqref{S1}. Using
\eqref{SS} and
\[
\left(\frac{z}{1+z}\right)^2
=1+\int_0^\infty (t-2) e^{-t} e^{-zt}\,dt,
\]
by simple transformations, we obtain
\[
\frac{z^2}{(1+z)^2}f(z)=
\int_0^\infty e^{-zt} r(t)\, dt,\quad z\in \mathbb C_{+},
\]
where
\begin{align*}
r(t):=&\int_0^\infty r_0(t,\tau)\,\mu(d\tau),\quad t \ge 0,
\\
r_0(t,\tau)
 :=
&
 te^{-t\tau}
+e^{-t}\int_0^t e^{-s\tau} e^s s(t-s-2)\,ds
\\
=&\begin{cases} \frac{\left(-2+(1-\tau)\tau t\right) \tau e^{-\tau t}
+\left(t+(2-t)\tau \right)e^{-t}}{(1-\tau)^3}, & \tau \neq 1,\\
te^{-t}(t^2/6 - t + 1), & \tau=1.
\end{cases}
\end{align*}
Moreover, it is not to hard to show  that
\begin{align*}
\int_0^\infty |r(t)|\,dt \le&
\int_0^\infty  \int_0^\infty |r_0(t,\tau)|\,dt\,\mu(d\tau)\\
\le& c\int_0^\infty \frac{\mu(d\tau)}{(1+\tau)^2}<\infty,
\end{align*}
for some constant $c>0.$ This leads to an alternative proof of Proposition \ref{regS}, (i).
\end{remark}

For $n \in \mathbb N$ and $t >0,$ denote
\begin{align}
(n-1)!\, L_{n,t}[m]:=& \int_0^\infty s^{n-1} e^{-s}\int_0^{t|1-s/n|} m(v)\, dv\, ds\label{LemL}\\
+&\int_0^\infty
\left|\int_0^{n(v/t+1)} s^{n-1}
e^{-s}\,[m(v+t-st/n)-m(v)]\,ds \right|\,dv.\notag
\end{align}

\begin{lemma}\label{compM}
Let $m$ be a positive  measurable function on $[0,\infty)$ such that its Laplace transform
$(\mathcal L m)(z)$ exists for every $z \in \mathbb C_+.$
Then for any $n \in \mathbb N$ and $t >0,$
\begin{equation}\label{repres0}
\Delta_{n,t}(z)(\mathcal L m)(z)=\int_0^\infty e^{-sz} q_{n,t}(s)\,ds,\quad z\in \mathbb C_{+},
\end{equation}
where $q_{n,t}(s)$ is a measurable on $[0,\infty)$ and
\begin{align*}
\int_0^\infty |q_{n,t}(s)|\,ds\le  L_{n,t}[m].
\end{align*}
\end{lemma}
\begin{proof}
By \eqref{rn}  we have for every $z\in \mathbb C_{+}$
\begin{equation*}
r_{n,t}(z)(\mathcal L m)(z) =\frac{1}{(n-1)!h^n}\int_0^\infty
e^{-uz} \left(\int_0^u s^{n-1} e^{-h^{-1}
s}\,m(u-s)\,ds\right)\,du.
\end{equation*}
Hence after a change of variable $s\mapsto s t/n$
\begin{equation*}\label{Rep10}
r_{n,t}(z)(\mathcal L m)(z)=\frac{1}{(n-1)!}\int_0^\infty e^{-uz}
\left(\int_0^{nu/t} s^{n-1} e^{-s}\,m(u-st/n)\,ds\right)\,du.
\end{equation*}
On the other hand,
\begin{equation*}\label{Rep20}
e^{-tz}(\mathcal L m)(z)=\int_0^\infty e^{-(u+t)z} m(u)\,du
=\int_t^\infty e^{-u z}\,m(u-t)\,du,\;\;z\in \mathbb C_{+}.
\end{equation*}

Then the above two formulas yield  \eqref{repres0}
with
\begin{equation}\label{repr10}
q_{n,t}(u)
:=\frac{1}{(n-1)!}\int_0^{nu/t} s^{n-1}
e^{-s}\,m(u-st/n)\,ds-\chi(u-t) m(u-t),
\end{equation}
where $\chi(\cdot)$ is the characteristic function of $[0,\infty).$

Taking into account (\ref{int1}) we transform (\ref{repr10}) further
to the form
\begin{align*}
q_{n,t}(u)=&\frac{\chi(t-u)}{(n-1)!}\int_0^{nu/t} s^{n-1}
e^{-s}\,m(u-st/n)\,ds\\
-&\frac{\chi(u-t)}{(n-1)!}\int_{nu/t}^\infty s^{n-1} e^{-s} m(u-t)\,ds
\label{repr100}\\
+&\frac{\chi(u-t)}{(n-1)!}
\int_0^{nu/t} s^{n-1}e^{-s}\,[m(u-st/n)-m(u-t)]\,ds.
\end{align*}
Hence
for any $n\in\mathbb N$ and $t>0,$
\begin{align*}
&\int_0^\infty |q_{n,t}(u)|\,du
\le \frac{1}{(n-1)!}\int_0^t\int_0^{nu/t} s^{n-1}
e^{-s}\,m(u-st/n)\,ds\,du
\\
+&\frac{1}{(n-1)!}\int_t^\infty \int_{nu/t}^\infty s^{n-1} e^{-s} m(u-t)\,ds\,du
\\
+&\frac{1}{(n-1)!}\int_t^\infty
\left|\int_0^{nu/t} s^{n-1}
e^{-s}\,[m(u-st/n)-m(u-t)]\,ds
\right|\,du
\\
=&\frac{1}{(n-1)!}\int_0^\infty s^{n-1} e^{-s}\int_0^{t|1-s/n|} m(v)\, dv \, ds
\\
+&\frac{1}{(n-1)!}\int_0^\infty
\left|\int_0^{n(v/t+1)} s^{n-1}
e^{-s}\,[m(v+t-st/n)-m(v)]\,ds \right|\,dv.
\end{align*}
The proof is complete.
\end{proof}
Let us illustrate Lemma \ref{compM} with several examples.

\begin{example}\label{Ex1}
$ a)$ Let
\[
f_1(z):=\frac{1}{z}=\int_0^\infty e^{-zv}\,dv, \qquad z \in \mathbb C_+.
\]
Then $f_1=\mathcal L m$ with $m(v)\equiv 1$ for $v \ge 0,$ and using
\eqref{int1} and \eqref{int20},
we have
\begin{align*}
L_{n,t}[m]=&\frac{1}{(n-1)!}\int_0^\infty s^{n-1} e^{-s} t|1-s/n| \, ds\\
\le& \frac{t}{(n-1)!}\left(\int_0^\infty s^{n-1} e^{-s} \, ds\right)^{1/2}
\left(\int_0^\infty s^{n-1} e^{-s} (1-s/n)^2 \, ds\right)^{1/2}\\
=&\frac{t}{\sqrt{n}}, \qquad n\in \mathbb N,\;\;t>0.
\end{align*}
Hence by Lemma \ref{compM},
\[
\|\Delta_{n,t} f_1\|_{{\rm A}^1_+(\mathbb C_+)}\le
\frac{t}{\sqrt{n}},\qquad n\in \mathbb N,\;\;t>0.
\]

$b)$  Let
\[
f_2(z):=\frac{1}{z^2}=\int_0^\infty e^{-zv} v\, dv, \qquad z \in \mathbb C_+,
\]
so that $f_2=\mathcal L m$ with $m(v)=v$ for $v\ge 0.$
By (\ref{zero}) we have
\begin{equation}\label{mod0}
\left|\int_0^{n(v/t+1)} s^{n-1}
e^{-s}\,(1-s/n)\, ds\right|
=\int_{n(v/t+1)}^\infty s^{n-1}
e^{-s}\,(s/n-1)\, ds,
\end{equation}
and then (see \eqref{int20})
\begin{align*}
L_{n,t}[m]
=&\frac{t^2}{2(n-1)!}\int_0^\infty s^{n-1} e^{-s} (1-s/n)^2\, ds\\
+&\frac{t}{(n-1)!}\int_0^\infty \int_{n(v/t+1)}^\infty  s^{n-1}
e^{-s}\,(s/n-1)\,ds \,dv
\\
\le&\frac{3t^2}{2(n-1)!}\int_0^\infty
 s^{n-1}
e^{-s}\,(1-s/n)^2\,ds \,dv\\
=&\frac{3t^2}{2n}, \qquad n\in \mathbb N,\;\;t>0.
\end{align*}
So, in this case, by Lemma \ref{compM},
\[
\|\Delta_{n,t}f_2\|_{{\rm A}^1_+(\mathbb C_+)}\le \frac{3
t^2}{2n},\qquad n\in \mathbb N,\quad t>0.
\]
c) Let
\[
f_3(z):=\left(1+\frac{\lambda}{z}\right)^2,\qquad \lambda>0, \quad z \in \mathbb C_+.
\]
Instead of identifying $m$ here we use the previous two examples.
Observe that by \eqref{rn} we have $\Delta_{n,t}\in {\rm A}^1_+(\mathbb C_+)$, and
\begin{equation}\label{deltaA}
\|\Delta_{n,t}\|_{{\rm A}^1_+(\mathbb C_+)}\le 2,\qquad n\in
\mathbb N,\quad t>0.
\end{equation}
By \eqref{deltaA} and Examples \ref{Ex1} a), b),
\begin{align*}
\|\Delta_{n,t}f_3\|_{{\rm A}^1_+(\mathbb C_+)}\le&
2+2\lambda \|\Delta_{n,t}f_1\|_{{\rm A}^1_+(\mathbb C_+)}+\lambda^2\|\Delta_{n,t}f_2\|_{{\rm A}^1_+(\mathbb C_+)}\\
\le& 2+\frac{2\lambda t}{\sqrt{n}}+\frac{3\lambda^2 t^2}{2n}\le
2\left(1+\frac{\lambda t}{\sqrt{n}}\right)^2, \qquad n\in \mathbb
N,\;\;t>0.
\end{align*}
\end{example}
\smallskip

The following technical lemma is crucial in the proof of the main result of this section, Theorem \ref{HF}.
We shift its proof to Appendix to clarify our presentation.
\begin{lemma}\label{LL}
Let  $\tau \ge 0,$ $t>0,$ and $n\in \mathbb N$ be fixed.
If
\begin{align}
Q_{n,t}^{(1)}(\tau):=&
\frac{1}{\tau^2}\int_0^\infty
s^{n-1}e^{-s}\,\bigl(1-(1+\tau t|1-s/n|)e^{-\tau t|1-s/n|}\bigr)\,ds, \label{L1lem}
\end{align}
and
\begin{align}\label{L2lem}
Q_{n,t}^{(2)}(\tau):
=& \int_0^\infty
e^{-\tau v} \psi_{n,t}(v) \, dv,\\
\psi_{n,t}(v):=&\left|\int_0^{n(v/t+1)} s^{n-1}
e^{-s}\,\bigl((v+t-st/n)e^{-\tau t(1-s/n)}-v \bigr)\, ds\right|,\notag
\end{align}
then
\begin{equation}\label{lnt}
Q_{n,t}(\tau):=\frac{Q_{n,t}^{(1)}(\tau)+Q_{n,t}^{(2)}(\tau)}{(n-1)!} \le \frac{12}{(\sqrt{n}/t+\tau)^2}.
\end{equation}
\end{lemma}
\smallskip

Lemma \ref{LL} implies the following key estimate for ${\rm A}^1_+(\mathbb C_+)$-norms of $\Delta_{n,t}f$ when $f \in {\mathcal S}_2.$
\begin{theorem}\label{HF}
Let $f\in \mathcal{S}_2$. Then $\Delta_{n,t}f  \in  {\rm A}^1_+(\mathbb C_+)$,
and
\begin{equation}
\|\Delta_{n,t}f \|_{{\rm A}^1_+(\mathbb C_+)}\le\,12 f(\sqrt{n}/t),\qquad n\in \mathbb N,\quad t>0.
\label{HFf}
\end{equation}
\end{theorem}

\begin{proof}
According to \eqref{deltaA}
it suffices to consider
$f$  of the form \eqref{S1}.
Then
\begin{equation}\label{repdelta}
\Delta_{n,t}(z)f(z)=
\int_0^\infty \frac{\Delta_{n,t}(z)}{(z+\tau)^2}\,\mu(d\tau), \qquad z \in \mathbb C_+.
\end{equation}
For fixed  $\tau\ge 0$ define
\begin{equation*}\label{Exp}
\varphi_\tau(z):=\frac{1}{(z+\tau)^2}=\int_0^\infty e^{-zv} m(v,\tau)\,du,\quad
m(v,\tau):=ve^{-\tau v}, \quad z \in \mathbb C_+.
\end{equation*}
Using Lemma \ref{compM} with
$m(v,\tau)=v e^{-\tau v}$
and noting that
\[
\int_0^w v e^{-v} dv= 1-(1+w)e^{-w},
\]
we obtain
\begin{equation*}\label{calc1}
L_{n,t}[m(\cdot,\tau)]
=Q_{n,t}(\tau),\quad n\in \mathbb N,\quad t>0,\quad \tau\ge 0,
\end{equation*}
where $L_{n,t}[m(\cdot, \tau)]$ is defined by \eqref{LemL} and $Q_{n,t}(\tau)$ is given by \eqref{lnt}.
So, by  Lemmas  \ref{compM}   and \ref{LL}
we have
\begin{equation}\label{12a}
\|\Delta_{n,t}\varphi_\tau\|_{{\rm A}^1_+(\mathbb C_+)}\le Q_{n,t}(\tau)\le
\frac{12}{(\sqrt{n}/t+\tau)^2}.
\end{equation}
Observe further that, in view of \eqref{deltaA},
$\tau \mapsto \Delta_{n,t} \varphi_\tau$ is a continuous, ${\rm A}^1_+(\mathbb C_+)$-valued function on $(0,\infty).$
Moreover, by \eqref{12a}, $\tau \mapsto \|\Delta_{n,t}\varphi_\tau\|_{{\rm A}^1_+(\mathbb C_+)}$ is Lebesgue integrable on $[0,\infty)$  for any $t >0$
and $n \in \mathbb N.$
Thus the ${{\rm A}^1_+(\mathbb C_+)}$-valued  Bochner integral
\begin{equation*}
\int_0^\infty \Delta_{n,t}\varphi_\tau\,\mu(d\tau)
\end{equation*}
is well-defined. Since point evaluations are bounded functionals on ${\rm A}^1_+(\mathbb C_+)$ and separate elements of ${\rm A}^1_+(\mathbb C_+),$
\eqref{repdelta} implies that the integral coincides with $\Delta_{n,t}f.$
Then by
\eqref{repdelta}, \eqref{12a},
and a standard  inequality for Bochner integrals
(see e.g. \cite[Theorem 3.7.6]{HilPhi}) we obtain for any $n\in \mathbb N$ and $t>0:$
\[
\|\Delta_{n,t}f\|_{{\rm A}^1_+(\mathbb C_+)}
\le \int_0^\infty \|\Delta_{n,t}\varphi_\tau\|_{{\rm A}^1_+(\mathbb C_+)}\,\mu(d\tau)\le \,12 f(\sqrt{n}/t).
\]
\end{proof}

\section{Main results}

Theorem \ref{HF} and \eqref{hillestimate}
imply immediately  the following statement.

\begin{theorem}\label{th1}
Let  $-A$ be the generator of a $C_0$-semigroup
$(e^{-tA})_{t\ge 0}$ on a Banach space $X,$ and let $\overline{{\rm ran}}\, (A)=X.$  Assume that
\[
M:=\sup_{t\ge 0}\,\|e^{-tA}\|\,<\infty.
\]
If  $f\in \mathcal{S}_2,$ then for any $x\in X,$
\begin{equation}
\|f(A)\Delta_{n,t}(A)x\|
\le\,  12 M\|x\| f(\sqrt{n}/t),\qquad n\in\mathbb N,\quad t>0,
\label{Main0}
\end{equation}
and for any $x=f(A)y$, $y\in {\rm dom}\,(f(A))$,
\begin{equation}
\|\Delta_{n,t}(A)x\|
\le\,  12 M\|y\| f(\sqrt{n}/t),\qquad n\in\mathbb N,\quad t>0.
\label{Main1}
\end{equation}
 In particular, if
$f(z)=z^{-\alpha}$, $\alpha\in (0,2]$,
then for every $x\in {\rm dom}\,(A^\alpha),$
\begin{equation}
\|\Delta_{n,t}(A)x\|
\le 12M\|A^\alpha x\|\left(\frac{t}{\sqrt{n}}\right)^\alpha, \qquad n\in\mathbb N,\quad t>0.
\label{Main2}
\end{equation}
\end{theorem}

Let us consider now the case when $-A$ generates a bounded $C_0$-semigroup
but the range of $A$ may not be dense, so that $A$ may not be injective. We will need the next approximation result.
\begin{theorem}\label{operS2}
Let $-A$ be the generator of a  bounded $C_0$-semigroup on a Banach space $X$, and let $f \in \tilde{\mathcal{S}}_2, f \neq 0.$ If $g=1/f,$ then for every $x\in {\rm dom}\,(A^2),$
\begin{equation}\label{limit}
\lim_{\delta\to 0+}\,g(A+\delta)x=g(A)x.
\end{equation}
\end{theorem}

\begin{proof}
Recall that by Proposition \ref{regS}  ${\rm dom}\,(A^2)$ is a core for $g(A+\delta),$ $\delta\ge 0$.
 Let $x\in {\rm dom}\,(A^2).$ Then for every $\delta >0$ there exists $y_\delta \in X$
such that $x=(A+1+\delta)^{-2}y_\delta$.
Note that
\[
g(A+\delta)x=g(A+\delta)(A+1+\delta)^{-2}y_\delta=
[g(z+\delta)\cdot (z+1+\delta)^{-2}](A)y_\delta.
\]
Moreover $g(z)/(z+1)^2\in {\rm A}^1_+(\mathbb C_+)$ and
hence
\[
\lim_{\delta\to 0+}\,\frac{g(\cdot+\delta)}{(\cdot+1+\delta)^2}=
\frac{g(\cdot)}{(\cdot+1)^2}.
\]
in the Banach algebra ${\rm A}^1_+(\mathbb C_+).$
Since
\[
\lim_{\delta\to 0+}\,y_\delta = \lim_{\delta \to 0+} (A+1+\delta)^{2}x=(A+1)^2x,
\]
we have
\begin{align*}
\lim_{\delta\to 0+}\,g(A+\delta)x=&
[g(z)\cdot(z+1)^{-2}](A)(A+1)^2x\\
=&g(A)(A+1)^{-2}(A+1)^2x\\
=&g(A)x.
\end{align*}
\end{proof}

Theorem \ref{operS2} allows us to adopt Theorem \ref{th1} to the case when the range of the generator may not be dense.

\begin{corollary}\label{CorM0}
Let $-A$ be the generator
of a  $C_0$-semigroup
$(e^{-tA})_{t\ge 0}$ on a Banach space $X$. Assume that
\[
M:=\sup_{t\ge 0}\,\|e^{-tA}\|\,<\infty.
\]
If  $g=1/f$, where $f\in \tilde{\mathcal {S}}_2, f \neq0$,
then for every $x\in {\rm dom}\,(g(A)),$
\[
\|\Delta_{n,t}(A)x\|
\le\,  12 M\frac{\|g(A)x\|}{g(\sqrt{n}/t)},\qquad n\in\mathbb N,\quad t>0.
\]
In particular, if
$g(z)=z^{\alpha}$, $\alpha\in (0,2]$,
then  (\ref{Main2}) holds.
\end{corollary}
\begin{proof}

Note that for any $\delta>0$
one has  ${\rm ran}\,(A+\delta)=X$. Thus
$f(A+\delta )$  is well defined  and bounded on $X$, moreover
\[
f(\delta+A)=f_\delta(A),
\]
where
$
f_\delta(z):=f(z+\delta), z>0.
$
For $f\in \tilde{\mathcal{S}}_2, f \neq 0,$ define
\[
g(z):=\frac{1}{f(z)},\;\;g_\delta(z):=\frac{1}{f_\delta(z)},\;\;\delta>0.
\]
Then, by the product rule \eqref{hpfc.e.prod},
\begin{equation}
f_\delta(A)g_\delta(A)x= g_\delta(A)f_\delta(A)x=x,\quad x\in {\rm dom}\,(A^2).
\label{inverse}
\end{equation}

From Theorem \ref{th1} and (\ref{inverse}) it follows that if $\delta>0$ and $x\in {\rm dom}\,(A^2)$
then for any $n\in\mathbb N$ and $t>0,$
\[
\|\Delta_{n,t}(A)x\|
\le\,  12 M
f_\delta(\sqrt{n}/t)
\|g_\delta(A)x\| \le 12 M f(\sqrt{n}/t)\|g_\delta(A)x\|.
\]
Let $\delta\to 0+$ in the above inequality.
Since ${\rm dom}\,(A^2)$ is core for $g(A)$,
\eqref{limit} implies the  statement.
\end{proof}

We finish this section with the estimate of convergence rate in Euler's formula for Komatsu's spaces $D_\infty^\alpha$ (defined in Introduction).

\begin{theorem}\label{ThmInt}
Let $-A$ be the generator of a $C_0$-semigroup $(e^{-tA})_{t\ge 0}$ on a Banach space $X.$
Suppose that
\[
M:=\sup_{t\ge 0}\,\|e^{-tA}\|\,<\infty.
\]
For any $\alpha\in (0,2]$ and $x\in D_\infty^\alpha$,
\begin{equation}\label{ApSt}
\|\Delta_{n,t}(A)x\|
\le\,  8 M\left(\frac{t}{\sqrt{n}}\right)^\alpha \|x\|_{D_\infty^\alpha},\qquad n\in\mathbb N,\quad t>0.
\end{equation}
\end{theorem}

\begin{proof}
For fixed $\lambda>0$ define
\[
f_\lambda(z):=\left(1+\frac{\lambda}{z}\right)^2,\quad
g_\lambda(z):=1/f_\lambda(z)=\left(\frac{z}{\lambda+z}\right)^2, \quad z >0.
\]
Note that $f_\lambda \in \tilde{\mathcal{S}}_2.$ Using Corollary \ref{CorM0} with
$g=g_\lambda$ and taking into account Example \ref{Ex1}, c),
we have for any $x\in X$ and $n \in \mathbb N$
\begin{align}\label{App1}
\|\Delta_{n,t}(A)x\|
\le&\,  2 M \left(1+\frac{\lambda t}{\sqrt{n}}\right)^2\|[A (A+\lambda)^{-1}]^2 x\|\\
=& 2 M \lambda^\alpha \|[A (A+\lambda)^{-1}]^2 x\|\frac{(\lambda t/\sqrt{n}+1)^2}{\lambda^\alpha},
\quad  t >0.\notag
\end{align}
Setting now $\lambda=\sqrt{n}/t$
in \eqref{App1} it follows that
\[
\|\Delta_{n,t}(A)x\|
\le 8 M \left(\frac{t}{\sqrt{n}}\right)^\alpha \sup_{s>0}\,
\bigl(s^\alpha \|[A (A+s)^{-1}]^2 x\|\bigr),\quad x\in D_\infty^\alpha,
\]
and \eqref{ApSt} holds.

\end{proof}

\section{Optimality of rates}

In this section we show that our estimates for convergence rates in Euler's approximation formula are in a sense optimal.
We will need the next estimate
for Stieltjes functions proved independently in many papers. It seems the earliest reference is \cite[Lemma 2]{Pust}.
\begin{lemma}\label{S}
If $f\in \mathcal{S},$  then
\begin{equation}
\label{ineqf2}
f(s)\le \sqrt{2} |f(\pm i s)|,\quad s>0.
\end{equation}
\end{lemma}

It will also be convenient to single out an auxiliary inequality involving  $\Delta_{n,t}.$
\begin{lemma}\label{log}
For any $n \in \mathbb N$ and $t>0,$
\begin{equation}\label{delta}
|\Delta_{n,t}(\pm i \sqrt{n}/t)|\ge 1-\frac{1}{\sqrt{2}}.
\end{equation}
\end{lemma}
\begin{proof}
For any $n \in \mathbb N$ and $t>0,$
\begin{align*}
|\Delta_{n,t}(\pm i \sqrt{n}/t)|
=\left|\frac{1}{(1+i/\sqrt{n})^n}-e^{-i\sqrt{n}}\right|
\ge 1-|1+i/\sqrt{n}|^{-n}.
\end{align*}
Since
\[
\log(1+t)\ge t\log 2,\quad t\in [0,1],
\]
it follows that
\[
|1+i/\sqrt{n}|^n= e^{(n/2)\log (1+1/n)}\ge
e^{(1/2)\log 2}=\sqrt{2},\quad n \in \mathbb N,
\]
and this yields \eqref{delta}.
\end{proof}

The result below shows that Theorem \ref{th1} and Corollary \ref{CorM0} are sharp if the spectrum of the generator is large enough.
\begin{theorem}\label{sharp}
Let  $-A$ be the generator
of a bounded $C_0$-semigroup
$(e^{-tA})_{t\ge 0}$ on a Banach space $X$.
Suppose that $\overline{{\rm ran}}\, (A)=X$ and
\begin{equation}\label{spectrum}
\{|s|:\,s\in \mathbb R,\,is\in \sigma(A)\}=\mathbb R_{+}.
\end{equation}
If $f\in \tilde{\mathcal{S}}_2,$  then
\begin{equation}
\| f(A)\Delta_{n,t}(A)\|\ge c f(\sqrt{n}/t),\;\;n\in\mathbb N,\;\;t>0,\;\;\;
c=\frac{1}{2}\;\left(1-\frac{1}{\sqrt{2}}\right).
\label{EstS1}
\end{equation}
In particular,
\[
\|A^\alpha \Delta_{n,t}(A)\|\ge c
(t/\sqrt{n})^\alpha,\;\;n\in\mathbb N,\;\;t>0,
\]
for any $\alpha\in (0,2].$
\end{theorem}

\begin{proof}
Let $n \in \mathbb N$ and $t>0$ be such that
$i\sqrt{n}/t\in\sigma(A).$ By the spectral inclusion theorem for
the Hille-Phillips functional calculus (see e.g. \cite[Theorem
16.3.5]{HilPhi} or \cite[Theorem 2.2]{GHT}) we obtain for every
$t>0:$
\begin{align}
\| f(A)\Delta_{n,t}(A)\|=&
\|(f\cdot \Delta_{n,t})(A)\|
\ge \sup_{\lambda \in \sigma(A)}\,|(f\cdot \Delta_{n,t})(\lambda)|\label{ShA}\\
\ge&
|(f\cdot \Delta_{n,t})(i \sqrt{n}/t)|=
|\Delta_{n,t}(i \sqrt{n}/t)|\cdot |f(i\sqrt{n}/t)|.\notag
\end{align}
Then, by Lemmas \ref{S} and \ref{log}, \eqref{ShA}
implies  \eqref{EstS1}.

If $-i\sqrt{n}/t\in\sigma(A)$ then, by Lemmas \ref{S} and  \ref{log}, the argument completely analogous to the above gives the same estimate \eqref{EstS1}.
\end{proof}

The assumptions of Theorem \ref{sharp} can trivially be satisfied as the next simple example shows.
\begin{example}\label{sam}
Let $X=L^2(\mathbb R_+).$
Define
\[
(Au)(s):=is u(s),\qquad u\in L^2(\mathbb R_+),
\]
with the maximal domain.
Then $-A$ generates a  $C_0$-semigroup $(e^{-At})_{t \ge 0}$ given by
$(e^{-At} u)(s)=e^{-is t}u(s),$ $t \ge 0,$
on $X,$ and
$\sigma(A)=i\mathbb R_+$. Thus, $A$ satisfies the conditions
of  Theorem \ref{sharp}.
\end{example}

The following statement complementing Theorem \ref{sharp} can be proved in the same way as Theorem \ref{sharp}.

\begin{theorem}\label{sharp1}
Let $-A$ be the generator
of a bounded $C_0$-semigroup
$(e^{-tA})_{t\ge 0}$ on  a Banach space $X$.
Suppose that $\overline{{\rm ran}}\,(A)=X$ and
$
\sigma(A)\cap  i\mathbb R
$
has an accumulation point at infinity.
If $f\in \tilde{\mathcal {S}}_2,$  then
\[
\limsup_{\sqrt{n}/t\to\infty}\,\frac{\|f(A)\Delta_{n,t}(A)\|}{f(\sqrt{n}/t)}>0.
\]
\end{theorem}

Finally, we show that Theorem \ref{th1} and Corollary \ref{CorM0} are sharp in a slightly weaker sense than in Theorem \ref{sharp}
but with no restriction on the spectrum of the generator.
\begin{corollary}\label{CorSh}
Let $A$  and $f$ satisfy the assumptions Theorem \ref{sharp1}.
Suppose in addition that
\[
\lim_{\tau\to\infty}\, \tau^2 f(\tau)=\infty.
\]
Then, whenever $\epsilon: (0,\infty)\mapsto (0,\infty)$
is a decreasing function
with  $\lim_{\tau\to\infty}\,\epsilon(\tau)=0$, there exists
$y\in {\rm ran}\,(f(A))$ such that
\begin{equation}\label{y0}
\limsup_{\sqrt{n}/t\to\infty}\,
\frac{\| \Delta_{n,t}(A)y\|}{\epsilon(\sqrt{n}/t)f(\sqrt{n}/t)}=\infty.
\end{equation}
\end{corollary}

\begin{proof}
By the  Theorem \ref{sharp1}
we have
\[
\limsup_{\sqrt{n}/t\to\infty}\,
\frac{\|f(A)\Delta_{n,t}(A)\|}{\epsilon(\sqrt{n}/t)f(\sqrt{n}/t)}=\infty,
\]
Since ${\rm ran}\,(A^2)\subset {\rm dom}\,(f(A))$, the product
rule \eqref{hpfc.e.prod} implies that $f(A)\Delta_{n,t}(A)$ is
similar to its restriction to ${\rm dom}\,(A^{-2})={\rm
ran}\,(A^2)$ by means of the isomorphism $A^2(I+A)^{-2}: X\mapsto
{\rm ran}\,(A^2).$ Then the uniform boundedness principle yields
$x\in {\rm ran}\,(A^2)\subset {\rm dom}\,(f(A))$ such that
\[
\limsup_{\sqrt{n}/t\to\infty}\, \frac{\|
f(A)\Delta_{n,t}(A)x\|_{{\rm
dom}\,(A^{-2})}}{\epsilon(\sqrt{n}/t)f(\sqrt{n}/t)}=\infty.
\]
On the other hand, setting $y=f(A)x,$ and using Example \ref{Ex1}
$b)$, we obtain
\begin{align}
\| f(A)\Delta_{n,t}(A)x\|_{{\rm dom}\,(A^{-2})}
=&\|\Delta_{n,t}(A)f(A)x\|+\|A^{-2}\Delta_{n,t}(A)f(A)x\|
\label{MMM}\\
=&\|\Delta_{n,t}(A)y\|+\|A^{-2}\Delta_{n,t}(A)y\|\notag \\
\le&  \|\Delta_{n,t}(A)y\|+\frac{3 M}{2}
\left(\frac{t}{\sqrt{n}}\right)^2\|y\|.\notag
\end{align}
Since $\tau^2 f(\tau)\to\infty$, $\tau\to\infty$, we may replace
$\epsilon(\tau)$ by $\max\{\epsilon(\tau), (\tau^2
f(\tau))^{-1}\}$, $\tau\ge 1,$ and suppose without loss of
generality that
\[
\beta:=\inf_{\tau\ge 1}\, \epsilon(\tau)\tau^2 f(\tau)>0.
\]
Hence, in view of
\[
\sup_{\sqrt{n}/t\ge
1}\,\frac{1}{\epsilon(\sqrt{n}/t)(\sqrt{n}/t)^2 f(\sqrt{n}/t)}\le
\frac{1}{\beta}<\infty,
\]
the statement follows from \eqref{MMM}.
\end{proof}

\begin{remark}\label{remarkonalpha}
Note that Theorem \ref{HF} does hot hold  for wider classes
$\mathcal{S}_\alpha$, $\alpha>2,$ of generalized Stieltjes
functions. Indeed, note that
\[
\lim_{z\to 0+}\frac{\Delta_{n,t}(z)}{z^2}=\lim_{z\to
0+}\,\frac{r_{n,t}(z)-e_t(z)}{z^2}=\frac{t^2}{2n}\not=0,\quad
\mbox{if}\quad t>0.
\]
Hence if $f\in \mathcal{S}_\alpha$, $\alpha>2$, is such that
$\lim_{z\to 0+}\,z^2f(z)=\infty$, and $\sigma(A)$ has accumulation
point at zero, then $f(A)\Delta_{n,t}(A)$ is not bounded.
(Otherwise, the inequality
$$
\| f(A)\Delta_{n,t}(A)\| \ge \sup_{\lambda \in
\sigma(A)}\,|(f\cdot \Delta_{n,t})(\lambda)|,
$$
leads to a contradiction.) Thus, \eqref{Main0} and \eqref{Main2}
are not true in this case.
\end{remark}

\section{Appendix}

\noindent
{\it Proof of Lemma \ref{LL}.}\\
\noindent It will  be convenient to denote
\begin{equation}
w:=w_{n,t}(s,\tau)=\tau t|1-s/n|. \label{wW}
\end{equation}
Using  \eqref{L2lem}  we obtain
\begin{align*}
&Q_{n,t}^{(2)}(\tau)\\
\le& \int_0^\infty e^{-\tau v}\int_0^{n(v/t+1)} s^{n-1}
e^{-s}\,\left|(v+t-st/n)e^{-\tau t(1-s/n)}-v \right| ds\,dv\\
=&\frac{1}{\tau^2}\int_0^n s^{n-1} e^{-s} \int_0^\infty e^{- v}
\left|(v+w)e^{-w}-v \right| dv\,ds\\
+&\frac{1}{\tau^2}\int_n^\infty s^{n-1} e^{-s} \int_w^\infty e^{-
v}
\left|(v-w)e^w -v\right| dv\,ds\\
=&\frac{1}{\tau^2 }\int_0^\infty s^{n-1} e^{-s}\int_0^\infty
e^{-v} |(1-e^{-w})v+we^{-w} |\,dv\,ds,
\end{align*}
and therefore
\begin{equation}
Q_{n,t}^{(2)}(\tau) \le \frac{1}{\tau^2}\int_0^\infty s^{n-1}
e^{-s}[(1-e^{-w})+we^{-w}]\,ds. \label{r00}
\end{equation}
Then, by (\ref{L1lem}) and (\ref{r00}),
\begin{equation}
Q_{n,t}(\tau)=\frac{Q^{(1)}_{n,t}(\tau)+Q_{n,t}^{(2)}(\tau)}{(n-1)!}
\le \frac{2}{(n-1)!{\tau^2}}\int_0^\infty s^{n-1} e^{-s}
[1-e^{-w}]\,ds \le \frac{2}{\tau^2}. \label{zerot}
\end{equation}

Now let us prove that
\begin{equation}
Q_{n,t}(\tau)\le \frac{3 t^2}{n},\;\;t\ge 0,\;\;n\in \mathbb
N,\;\;\tau>0. \label{form}
\end{equation}
Define
\[
q_{n,t}(v,s,\tau):=\left|v[e^{-\tau t(1-s/n)}-1+\tau t(1-s/n)] +
t(1-s/n)(e^{-\tau t(1-s/n)}-1)\right|.
\]
Then using  \eqref{mod0} we have
\begin{align*}
Q_{n,t}^{(2)}(\tau) \le& \int_0^\infty e^{-\tau v}
\int_0^{n(v/t+1)} s^{n-1} e^{-s}\,q_{n,t}(v,s,\tau)\,ds\,dv
\\
+&t\int_0^\infty e^{-\tau v}|1-\tau v|\int_{n(v/t+1)}^\infty
s^{n-1}
e^{-s}\,(s/n-1) ds\, dv\\
=&\frac{1}{\tau^2}\int_0^n s^{n-1} e^{-s}\int_0^\infty  e^{-v}
\tau q_{n,t}(v/\tau,s,\tau)\,dv\,ds\\
+&\frac{1}{\tau^2}\int_n^\infty s^{n-1} e^{-s}\int_w^\infty
e^{-v}
\tau q_{n,t}(v/\tau,s,\tau)\,dv\,ds\\
+&\frac{1}{\tau^2}\int_n^\infty  s^{n-1} e^{-s} wu(w)\,ds,
\end{align*}
where
\[
u(w):=\int_0^w |1-v|e^{-v}\,d v\le 1-e^{-w}.
\]

For $s\le n$ we have
\begin{align*}
\tau q_{n,t}(v/\tau,s,\tau)
=&\left|v[e^{-w}-1+w]+ w(e^{-w}-1)\right|\\
\le & v(e^{-w}-1+w)+w(1-e^{-w}),
\end{align*}
and similarly if  $s\ge n$ then
\begin{align*}
\tau q_{n,t}(v/\tau,s,\tau)
=&\left|v[e^{w}-1-w]-w(e^{w}-1)\right|\\
\le& v(e^{w}-1-w)+w(e^w-1).
\end{align*}
So,
\begin{align}
&\tau ^2 Q_{n,t}^{(2)}(\tau)\label{r20}\\
\le& \int_0^n s^{n-1} e^{-s} \int_0^\infty e^{-v}(v[e^{-w}-1+w]
+ w(1-e^{-w}))\,dv\,ds \notag\\
+&\int_n^\infty s^{n-1} e^{-s}\int_w^\infty e^{-v}
\left|v[e^{w}-1-w]+w(e^{w}-1)\right|\,dv\,ds\notag\\
+&\int_n^\infty  s^{n-1} e^{-s} w (1-e^{-w})\,ds\notag\\
=& \int_0^n s^{n-1} e^{-s}[2w-1+(1-w)e^{-w}]\,ds\notag \\
+&\int_n^\infty s^{n-1} e^{-s} [1+3w-(1+4w+w^2)e^{-w}] \,ds.\notag
\end{align}

Then, by (\ref{L1lem}) and (\ref{r20}), we infer that
\begin{align*}
\tau^2 (Q^{(1)}_{n,t}(\tau)+Q^{(2)}_{n,t}(\tau))
\le& 2\int_0^n s^{n-1}e^{-s} w(1-e^{-w})\,ds\\
+&\int_n^\infty s^{n-1} e^{-s}[2+3w-(w^2 +5w+2)e^{-w}]\,ds.
\end{align*}
Using now elementary inequalities
\[
w(1-e^{-w})\le w^2,\;\;\; 2+3w-(w^2+5w+2)e^{-w}\le 3w^2,\;\;\;w\ge
0,
\]
where $w^2=\tau^2 t^2(1-s/n)^2$ (see \eqref{wW}), we obtain that
\begin{equation}\label{second}
Q_{n,t}(\tau) \le \frac{3t^2}{(n-1)!}\int_0^\infty s^{n-1}
(1-s/n)^2\,ds= \frac{3t^2}{n},
\end{equation}
i.e. \eqref{form} holds.

Hence, from (\ref{zerot}), (\ref{second}) and the inequality
\[
\min\left\{\frac{1}{a^2},\frac{1}{b^2}\right\}\le
\frac{4}{(a+b)^2},\;\;\;a,b>0,
\]
it follows that
\[
Q_{n,t}(\tau)\le
3\min\left\{\frac{1}{\tau^2},\frac{t^2}{n}\right\} \le
\frac{12}{(\sqrt{n}/t+\tau)^2}, \qquad n \in \mathbb N, \,\, \tau
\ge 0, \,\, t >0.
\]


\begin{thebibliography}{99}

\bibitem{ABHN} W.~Arendt, C.~J.~K. Batty, M.~Hieber, and F.~Neubrander,
\emph{Vector-valued  {L}aplace {T}ransforms and {C}auchy
{P}roblems}, second edition, Monographs in Mathematics,
  \textbf{96}, Birkh\"auser/Springer, Basel, 2011.

\bibitem{BT79}
P. Brenner and V. Thom\'ee, \emph{On rational approximations of
semigroups}, SIAM J. Numer. Anal. \textbf{16} (1979), 683--694.



\bibitem{deLau95}
R. deLaubenfels, \emph{Automatic extensions of functional
calculi}, Studia Math. \textbf{114} (1995), 237--259.

\bibitem{Fiory}
S. Fiory, F. Neubrander and L. Weis, {\it Consistency and
stabilization of rational approximation schemes for
$C_0$-semigroups,} in: Evolution Equations; Applications to
Physics, Industry, Life Science and Economics (Levico Terme,
2000), Progr. Nonlinear Differential Equations Appl., \textbf{55},
Birkh\"auser, Basel, 2003, pp. 181--193.


\bibitem{Hassan}
H. Emamirad and A. Rougirel, \emph{A functional calculus approach
for the rational approximation with nonuniform partitions},
 Discrete Contin. Dyn. Syst. \textbf{22} (2008),  955--972.


\bibitem{GHT1}
A. Gomilko, M. Haase, and Yu. Tomilov, \emph{On rates in mean
ergodic theorems}, Math. Research Letters \textbf{18} (2011),
201--213.

\bibitem{GHT} A. Gomilko, M. Haase, and Yu. Tomilov, \emph{Bernstein functions and rates in mean ergodic theorems
for operator semigroups,} J. d'Analyse Mathematique \textbf{118}
(2012), 545--576.

\bibitem{GT} A. Gomilko and  Yu. Tomilov, \emph{What does a rate in a mean ergodic theorem imply?}, submitted.

\bibitem{Ha06} M. Haase, \emph{The Functional Calculus for Sectorial Operators,}
Operator Theory: Advances and Applications, \textbf{169},
Birkh\"auser,  Basel, 2006.

\bibitem{HK}
R. Hersh and  T. Kato, \emph{High-accuracy stable difference
schemes for well-posed initial value problems,} SIAM J. Numer.
Anal. \textbf{16} (1979), 670-682.

\bibitem{HilPhi}
E. Hille and R. S. Phillips, \emph{Functional Analysis and
Semi-Groups,} 3rd printing of rev. ed. of 1957, Colloq. Publ.
\textbf{31}, AMS, Providence, RI, 1974.

\bibitem{Karp}
D. Karp and E. Prilepkina, \emph{Generalized Stieltjes transforms:
basic aspects,}
 arXiv:1111.4271v2.

\bibitem{Komatsu1966}
H. Komatsu, \emph{Fractional powers of operators,} Pacific J.
Math. \textbf{19} (1966),  285--346.

\bibitem{Komatsu}
H. Komatsu, \emph{Fractional powers of operators. II.
Interpolation spaces,}
 Pacific J. Math. \textbf{21} (1967),  89--111.

\bibitem{KovDis}
M. Kov\'acs, \emph{On the qualitative properties and convergence
of time discretization
 methods for semigroups,} dissertation,
 Lousiana State University, 2004.



\bibitem{Kovacs07}
M. Kov\' acs, \emph{On the convergence of rational approximations
of semigroups on intermediate spaces}, Math. Comp. \textbf{76}
(2007),  273--286.



\bibitem{NewZ}
M. Kov\' acs and F. Neubrander, \emph{On the inverse
Laplace-Stieltjes transform of $A$-stable rational functions},
 New Zealand J. Math. \textbf{36} (2007), 41--56.

\bibitem{Mart1}
C. Martinez Carracedo and M. Sanz Alix, \emph{The Theory of
Fractional Powers of Operators,} North-Holland, Amsterdam, 2001.

\bibitem{Pust}
E. I. Pustyl'nik, \emph{Some inequalities in the theory of
selfadjoint operators}, Izv. Vyssh. Uchebn. Zaved. Mat.,
\textbf{23}(1979), no. 6,  52--57 (in Russian), transl. in Soviet
Mathematics (Izvestiya VUZ. Matematika) \textbf{23}(1979), no. 6,
54–-59.



\bibitem{SchilSonVon2010}
 R. Schilling, R. Song, and Z. Vondra{\v{c}}ek,  \emph{Bernstein functions},
 de Gruyter Studies in Mathematics, \textbf{37}, Walter de Gruyter, Berlin, 2010.



\end{thebibliography}
\end{document}